\documentclass[12pt]{article}
\usepackage{setspace} 
\usepackage[dvips]{graphicx} 
\usepackage{amsmath,amsthm,amsfonts}
\usepackage{epsfig}
\usepackage{subfig}
\RequirePackage[colorlinks,citecolor=blue,urlcolor=blue,linkcolor=blue]{hyperref}
\hypersetup{
colorlinks = true,
citecolor=blue,
urlcolor=blue,
linkcolor=blue,
pdfauthor = {Alexey Kuznetsov},
pdfkeywords = {truncated theta function, Mordell integral, Riemann zeta function},
pdftitle = {Computing the truncated theta function via Mordell integral},
pdfpagemode = UseNone
}
    \def\beq{\begin{eqnarray}}
    \def\eeq{\end{eqnarray}}
    \def\beqq{\begin{eqnarray}}
    \def\eeqq{\end{eqnarray}}

    \def\re{\textnormal {Re}}
    \def\im{\textnormal {Im}}

    \def\r{{\mathbb R}}

    \def\c{{\mathbb C}}

    \def\d{{\textnormal d}}
    \def\i{{\textnormal i}}

	\newtheorem{theorem}{Theorem}
	\newtheorem{lemma}{Lemma}
	\newtheorem{proposition}{Proposition}

\title{Computing the truncated theta function via Mordell integral}
\author{
A. Kuznetsov
\thanks{{Research supported by the
Natural Sciences and Engineering Research Council of Canada. The author would like to thank the anonymous referee for careful reading of the paper and for providing many valuable comments and suggestions. The first version of this paper was written while the author was visiting the Department of Mathematical Sciences of the University of Bath, whose hospitality is highly appreciated. 
}}  \\ \\
Dept. of Mathematics and Statistics\\  York University \\
4700 Keele Street 
\\Toronto, ON \\ M3J 1P3,  Canada 
 }

\begin{document}

\maketitle

\begin{abstract}
Hiary \cite{Hiary_theta} has presented an algorithm which allows to evaluate
 the truncated theta function $\sum_{k=0}^n \exp(2\pi \i (zk+\tau k^2))$ to within $\pm \epsilon$ in $O(\ln(\tfrac{n}{\epsilon})^{\kappa})$ arithmetic operations for any real $z$ and $\tau$. 
This remarkable result has many applications in Number Theory, in particular it is the crucial element
 in Hiary's algorithm for computing $\zeta(\tfrac{1}{2}+\i t)$ to within $\pm t^{-\lambda}$ 
in $O_{\lambda}(t^{\frac{1}{3}}\ln(t)^{\kappa})$ arithmetic operations, see \cite{Hiary_zeta}. 
We present a significant simplification of Hiary's algorithm for evaluating the truncated theta function. 
Our method avoids the use of the Poisson summation formula, and substitutes it with an explicit identity involving the Mordell integral. This results in an algorithm which is efficient, conceptually simple and easy to implement. 
\end{abstract}

{\vskip 0.15cm}
{\vskip 0.05cm}

\newpage


\section{Introduction and main results}\label{sec_introduction}

The truncated theta function  is defined  as 
\beq\label{def_F}
F_n(z,\tau):=\sum\limits_{k=0}^n e^{2\pi \i (zk+\tau k^2)},
\eeq
where $n\in {\mathbb N}$ and $z, \tau \in \r$. We are interested in computing efficiently
 the truncated theta function $F_n(z,\tau)$ 
for large values of $n$. 
The motivation for developing such algorithms comes from computational Number Theory. One particularly 
important application is for computing the Riemann zeta function
$\zeta(\tfrac{1}{2}+\i t)$ for large $t$. The classical result by Riemann and Siegel (see 
 formula 4.17.5 in \cite{Ti1986}) allows to evaluate $\zeta(\tfrac{1}{2}+\i t)$ to within $\pm t^{-\lambda}$ in $O_{\lambda}(t^{\frac{1}{2}})$ arithmetic operations (by an arithmetic  
operation we mean additions, multiplications, evaluations of the logarithm and of the complex exponential function). It is a very challenging task to prove that the number of operations can be reduced to  $O(t^{\alpha})$ with some $\alpha<1/2$. It is an even harder problem to develop an algorithm which would not only have theoretical complexity $O(t^{\alpha})$, but would also be feasible for practical implementation. 
 In the recent papers \cite{Hiary_zeta,Hiary_theta}, Hiary has developed two new algorithms, which allow us to evaluate 
 $\zeta(\tfrac{1}{2}+\i t)$ to within $\pm t^{-\lambda}$ in $O_{\lambda}(t^{\frac{1}{3}}\ln(t)^{\kappa})$ and $O_{\lambda}(t^{\frac{4}{13}}\ln(t)^{\kappa})$) arithmetic operations (we will refer to them as $1/3$-algorithm and $4/13$-algorithm). 
  The $1/3$-algorithm is particularly attractive compared with the $4/13$-algorithm, 
as it is conceptually simpler, has no substantial storage requirements and should be easier to implement.  
And this brings us back to the truncated theta functions, as the $1/3$-algorithm is 
fundamentally based on another result by Hiary (see  \cite[Theorem 1.1]{Hiary_theta}), which 
states that $F_n(z,\tau)$ and its  derivatives can be computed to within $\pm \epsilon$ in $O(\ln\left(\tfrac{n}{\epsilon}\right)^{\kappa})$ arithmetic operations. 

Let us summarize the main ideas behind Hiary's truncated theta function algorithm. The algorithm is based on the iterative application of the following two-step procedure. The first step is to ``normalize" $z$ and $\tau$ so that 
they lie in the intervals  $z \in [-1/2,1/2]$ and  $\tau \in  [-1/4,1/4]$. This is easily achieved 
via the identities
\beq\label{identities_step_1}
F_n(z,\tau)=F_n\left(z+\tfrac{1}{2},\tau+\tfrac{1}{2}\right)=F_n(z+k,\tau+l), 
\;\;\; k, l \in {\mathbb Z},
\eeq
both of which follow directly from \eqref{def_F}. If $|\tau| < 1/n$, then we compute $F_n(z,\tau)$ by the Euler-Maclaurin summation
(see Section 3.2 in \cite{Hiary_theta}), 
otherwise we proceed to the second step of the algorithm, which is based on the following identity
\beq\label{Hiary_main_formula}
F_n(z,\tau)=\frac{e^{\frac{\pi \i }{4}-\frac{\pi \i z^2}{2\tau}}}{\sqrt{2\tau}}  F_m\left(\tfrac{z}{2\tau},
-\tfrac{1}{4\tau}\right)+R_{m,n}(z,\tau),
\eeq
where $\tau>0$ and $m=\lfloor 2n \tau \rfloor$. Here and everywhere in this paper $\lfloor x \rfloor:=\max\{ k\in {\mathbb Z} \; : \; k \le x\}$ denotes the floor function. A corresponding identity for $\tau<0$ can be derived by taking the conjugate of both sides of 
\eqref{Hiary_main_formula} and using the fact that $\overline{F_n(z,\tau)}=F_n(-z,-\tau)$. 

Identity \eqref{Hiary_main_formula} plays the central role in Hiary's algorithm. It is derived by applying the Poisson summation formula to \eqref{def_F} and using the crucial self-similarity property of the Gaussian function (the fact that taking Fourier transform of 
 $\exp(-ax^2-bx-c)$ results in a function of the same form but with different parameters).  
The truncated theta function $F_m(\cdot,\cdot)$ in the right-hand side of \eqref{Hiary_main_formula}
is the dominant contribution arising from the Poisson summation formula, while $R_{m,n}(z,\tau)$ can be considered as the remainder term. 
The key idea behind the second step in Hiary's algorithm is that
 the identity \eqref{Hiary_main_formula} transforms a long sum with $n$ terms into a short
sum having $m=\lfloor 2n |\tau| \rfloor\le n/2$ terms (recall that we have normalized $\tau$ so that $|\tau|\le 1/4$). If $m$ is smaller than a fixed power of $\ln(n/\epsilon)$, then we compute $F_m(\cdot,\cdot)$ by direct summation, otherwise we apply the same two-step procedure to $F_m(\cdot,\cdot)$. It is clear that this algorithm will terminate after at most $\log_2(n)$ iterations (where $\log_2(\cdot)$ denotes the logarithm with base two). 

Our goal in this paper is to simplify Hiary's algorithm for computing the truncated theta function. In order to present our results, first we need to introduce {\it Mordell integral}, which is defined as 
\beq\label{def_hztau}
h(z,\tau):=\int_{\r}  \frac{e^{\pi \i \tau x^2 -2 \pi z x}}{\cosh(\pi x)} \d x, \;\;\; z\in \c, \; \im(\tau)>0. 
\eeq
We need to extend this definition for real values of $\tau$. Assume that $\tau$ lies in the quadrant $\re(\tau)>0$ and $\im(\tau)>0$. Then the following identity is true
\beq\label{def_hztau_2}
h(z,\tau)
=
e^{\frac{\pi \i}{4}} \int_{\r} 
\frac{e^{-\pi \tau y^2-2\pi z e^{\frac{\pi \i}{4}}  y}}{\cosh(\pi  e^{\frac{\pi \i}{4}} y)} \d y
=2e^{\frac{\pi \i}{4}} \int_0^{\infty} e^{-\pi \tau y^2}
\frac{\cosh(2\pi z e^{\frac{\pi \i}{4}}  y)}{\cosh(\pi  e^{\frac{\pi \i}{4}} y)} \d y.
\eeq
This result can be easily established by rotating the contour of integration $\r^+ \mapsto e^{\frac{\pi \i}{4}}\r^+$
in the integral \eqref{def_hztau} and then changing the variable of integration $x=e^{\frac{\pi \i}{4}}y$.
It is clear the integral in \eqref{def_hztau_2} provides the analytic continuation of $h(z,\tau)$ into the domain $\re(\tau)>0$, $z\in \c$. We will adopt this as the definition of $h(z,\tau)$ for $\tau>0$ and $z\in \r$, while for $\tau<0$ and $z\in \r$ we will define $h(z,-\tau):=\overline{h(z,\tau)}$. 

Mordell integral and related functions have been studied for more than a century. One particular example of such 
integrals was used by Riemann to derive the functional equation for the zeta function (see section 2.10 in \cite{Ti1986}), while some further special cases were investigated by Ramanujan \cite{Ramanujan}. 
Mordell \cite{Mordell,Mordell2} has developed a general theory of 
$h(z,\tau)$ and some related integrals. Among many other results, Mordell has discovered the connections
between  $h(z,\tau)$ and theta functions, he has studied general modular transformations of $h(z,\tau)$ and has proved that $h(z,\tau)$ can be expressed in terms of the Gauss sums when $\tau$ is rational (more precisely,
$h(z,p/q)$ can be expressed in terms of functions $F_q(z,p/(2q))$ and $F_p(z,-q/(2p))$, see section 8 in 
\cite{Mordell2}). Recently,  the Mordell integral $h(z,\tau)$ was investigated by Zwegers in his Ph.D. thesis \cite{Zwegers}, and we will follow Zwegers's notation in our paper.

The following theorem is our first main result, and it is the basis for our
simplified version of Hiary's algorithm for fast evaluations of $F_n(z,\tau)$. 
\begin{theorem}\label{thm_main}
For $\tau>0$, $z\in \r$ and $m, n \in {\mathbb N}\cup\{0\}$
\beq\label{main_formula}
F_n(z,\tau)=\frac{e^{\frac{\pi \i }{4}-\frac{\pi \i z^2}{2\tau}}}{\sqrt{2\tau}}  F_m\left(\tfrac{z}{2\tau},
-\tfrac{1}{4\tau}\right)+R_{m,n}(z,\tau),
\eeq
where
\beq\label{def_Rmnztau}
R_{m,n}(z,\tau)&=& -\frac{\i}{2}e^{-\pi \i \left(z-\frac{\tau}{2} \right) } h\left(z-\tau+\tfrac{1}{2},-2\tau \right)
\\ \nonumber
&-&\frac{\i}{2} (-1)^m e^{2\pi \i\left(n+\frac{1}{2}\right)  \left(z+ \tau (n+\frac{1}{2}) \right)} 
h\left(z+(2n+1)\tau-m-\tfrac{1}{2},-2\tau\right).
\eeq
\end{theorem}

Using Theorem \ref{thm_main} we are able to establish the following result (which should be compared with Theorem 1.1 in \cite{Hiary_theta}).
\begin{theorem}\label{thm_main2}
There exists an algorithm such that for any  $\epsilon \in (0,\tfrac{1}{10})$, any $z,\tau \in (0,1)$ 
and any integer $n\ge 1$ the value of the function $F_n(z,\tau)$ can be evaluated to within $\pm \epsilon$ using 
$\le C_1 \ln\left(\tfrac{n}{\epsilon}\right)^3$ arithmetic operations on numbers of $\le C_2 \ln\left(\tfrac{n}{\epsilon}\right)$ bits. The algorithm requires $\le C_3 \ln\left(\tfrac{n}{\epsilon}\right)^2$ bits of memory. 
\end{theorem}
In the above result (and everywhere in this paper) we assume that $A_1,A_2,\dots$ and $C_1,C_2,\dots$ are constants, which are positive and absolute, in the sense that they do not depend on the values of any other parameters.

The paper is organized as follows: Section \ref{sec_proofs} contains the proofs of Theorems \ref{thm_main} and \ref{thm_main2} and in Section \ref{section_numerics} we discuss the practical implementation  and some extensions of the algorithm.

\section{Proof of the main results}\label{sec_proofs}

\begin{lemma}\label{lemma_h_z_tau}
For $z \in \r$ and $\tau>0$ 
\beq\label{hztau_period_1}
&&h(z,\tau)+h(z+1,\tau)=\frac{2}{\sqrt{\tau}}e^{\frac{\pi \i }{4}+\frac{\pi \i}{\tau} (z+\frac{1}{2})^2},\\
\label{hztau_1_over_tau}
&&h(z,\tau)=\frac{1}{\sqrt{\tau}}e^{\frac{\pi \i }{4}+\frac{\pi \i z^2}{\tau}}h(\tfrac{z}{\tau}, -\tfrac{1}{\tau}).
\eeq
\end{lemma}
\begin{proof}
While the above identities are not new (see \cite{Mordell2} and  \cite[Proposition 1.2]{Zwegers}), we present the sketch of the proof for the sake of completeness.  Let us denote $\theta:=e^{\frac{\pi \i }{4}}$ and assume that $\tau>0$. From the second integral representation in \eqref{def_hztau_2} we obtain
\beqq
h(z)+h(z+1)=\theta \int_{\r} \frac{e^{-\pi \tau y^2-2\pi z \theta y}}{\cosh(\pi \theta y)} \left(1+e^{-2\pi \theta y} \right) \d y=
2\theta \int_{\r} e^{-\pi \tau y^2 - 2\pi \theta y \left(z+\frac{1}{2} \right)} \d y. 
\eeqq
Evaluating the integral in the right-hand side of the above identity (use formula (3.323.2) in \cite{Jeffrey2007}) gives us \eqref{hztau_period_1}. 

In order to prove \eqref{hztau_1_over_tau} we use formulas (3.323.2) and (3.511.4) in \cite{Jeffrey2007} and evaluate the following two integrals: For $w\in \r$
\beqq
\int_{\r} e^{-\pi \tau y^2-2\pi z \theta y-2\pi \i y w} \d y=\frac{1}{\sqrt{\tau}} e^{\frac{\pi \i}{\tau} (z+\theta w)^2},  \;\;\;\;\;
\theta \int_{\r} \frac{e^{-2\pi \i y w}}{\cosh(\pi \theta y)} \d y=\frac{1}{\cosh(\pi \bar\theta w)}.
\eeqq
Applying Parseval's Theorem for Fourier transform to the first integral in \eqref{def_hztau_2} gives us
\beqq
h(z,\tau)=\frac{1}{\sqrt{\tau}} \int_{\r}  
\frac{e^{\frac{\pi \i}{\tau} (z+\theta w)^2}}{\cosh(\pi \bar \theta w)} \d w
=\frac{1}{\sqrt{\tau}}e^{\frac{\pi \i }{4}+\frac{\pi \i z^2}{\tau}} \overline{h(z,\tfrac{1}{\tau})}=
\frac{1}{\sqrt{\tau}}e^{\frac{\pi \i }{4}+\frac{\pi \i z^2}{\tau}}h(\tfrac{z}{\tau}, -\tfrac{1}{\tau}).
\eeqq
\end{proof}

\vspace{0.15cm}
\noindent
{\bf Proof of Theorem \ref{thm_main}:}
We take the conjugate of both sides of equation \eqref{hztau_period_1} and obtain
\beqq
h\left(z-\tfrac{1}{2},-\tau\right)=
-h\left(z+\tfrac{1}{2},-\tau\right)+\frac{2}{\sqrt{\tau}}e^{-\frac{\pi \i }{4}-\frac{\pi \i z^2}{\tau}}.
\eeqq
Iterating the above identity $m$ times gives us
\beqq
h\left(z-\tfrac{1}{2},-\tau\right)=
(-1)^{m+1} 
h\left(z+m+\tfrac{1}{2},-\tau\right)+\frac{2}{\sqrt{\tau}} 
\sum\limits_{k=0}^m (-1)^k e^{-\frac{\pi \i}{4}-\frac{\pi \i}{\tau} (z+k)^2},
\eeqq
which is equivalent to
\beq\label{proof_eq2}
h\left(z-\tfrac{1}{2},-\tau\right)=
(-1)^{m+1} 
h\left(z+m+\tfrac{1}{2},-\tau\right)+\frac{2}{\sqrt{\tau}} 
e^{-\frac{\pi \i}{4}-\frac{\pi \i z^2}{\tau}}
F_m\left(\tfrac{1}{2}-\tfrac{z}{\tau},-\tfrac{1}{2\tau}\right).
\eeq
We change variables $z=w/t$, $\tau=-1/t$ and $m=n$ in \eqref{proof_eq2}, and then
apply transformation \eqref{hztau_1_over_tau} to both $h$-functions, which results in 
the following identity
\beq\label{proof_eq3}
e^{\frac{\pi \i}{4}+\frac{\pi \i}{t}\left(w-\frac{t}{2}\right)^2} h\left(w-\tfrac{t}{2},-t\right)
&=&(-1)^{n+1} e^{\frac{\pi \i}{4}+\frac{\pi \i}{t}\left(w+nt+\frac{t}{2}\right)^2} 
h\left(w+nt+\tfrac{t}{2},-t\right)\\ \nonumber
&+&
2
e^{\frac{\pi \i}{4}+\frac{\pi \i w^2}{t}}
F_n\left(w-\tfrac{1}{2},\tfrac{t}{2}\right).
\eeq
At the same time, changing variables $z=w+nt-m+t/2-1/2$ and $\tau=t$ in \eqref{proof_eq2} we obtain  
\beq\label{proof_eq4}
h\left(w+nt-m+\tfrac{t}{2}-1,-t\right)&=&
(-1)^{m+1} h\left(w+nt+\tfrac{t}{2},-t\right)\\
\nonumber
&+&
\frac{2}{\sqrt{t}} 
e^{-\frac{\pi \i}{4}-\frac{\pi \i}{t}\left(w+nt-m+\frac{t}{2}-\frac{1}{2}\right)^2}
F_m\left(\tfrac{1}{t} \left(m+\tfrac{1}{2}-w\right),-\tfrac{1}{2t}\right). 
\eeq
Eliminating $h\left(w+nt+\tfrac{t}{2},-t\right)$ from the two equations \eqref{proof_eq3} and \eqref{proof_eq4} 
gives us
\beqq
&&e^{\frac{\pi \i}{4}+\frac{\pi \i}{t}\left(w-\frac{t}{2}\right)^2} h\left(w-\tfrac{t}{2},-t\right)=2
e^{\frac{\pi \i}{4}+\frac{\pi \i w^2}{t}}
F_n\left(w-\tfrac{1}{2},\tfrac{t}{2}\right)+(-1)^{n+1} e^{\frac{\pi \i}{4}+\frac{\pi \i}{t}\left(w+nt+\frac{t}{2}\right)^2} 
\\  \nonumber
&\times& 
(-1)^{m+1}
\left[h\left(w+nt-m+\tfrac{t}{2}-1,-t\right)-\frac{2}{\sqrt{t}} 
e^{-\frac{\pi \i}{4}-\frac{\pi \i}{t}\left(w+nt-m+\frac{t}{2}-\frac{1}{2}\right)^2}
F_m\left(\tfrac{1}{t} \left(m+\tfrac{1}{2}-w\right),-\tfrac{1}{2t}\right) \right].
\eeqq
The above identity is equivalent to \eqref{main_formula}. In order to verify this one would need 
to apply the transformation
\beqq
F_m\left(\tfrac{1}{t} \left(m+\tfrac{1}{2}-w\right),-\tfrac{1}{2t}\right)=
e^{\frac{\pi \i}{t} (m^2-m(2w-1))}F_m\left(\tfrac{1}{t} \left(w-\tfrac{1}{2}\right),-\tfrac{1}{2t}\right),
\eeqq
(which follows easily from \eqref{def_F} by changing the index of summation $k\mapsto m-k$),
then change the variables $w=z+\tfrac{1}{2}$ and $t=2\tau$ and simplify the result.  
\qed

\vspace{0.35cm}
Now we need to introduce several new objects. We define the sequence $\{\tilde E_k\}_{k\ge 0}$ as 
\beq\label{def_Bk_Ek2}
 \frac{1}{\cosh(x)}=\sum\limits_{k\ge 0} \tilde E_k x^k,
\eeq
or, alternatively,  $\tilde E_k=E_k/k!$ where $\{E_k\}_{k\ge 0}$ are  Euler numbers.
We define the function $f_{\tau}(z)$ and the sequence of polynomials $\{q_{k}(\tau)\}_{k\ge 0}$ as 
\beq\label{def_f_tau_z}
f_{\tau}(z):=\frac{e^{\frac{\i \tau z^2}{\pi}}}{\cosh(z)}=\sum\limits_{k\ge 0} q_{k}(\tau) z^{2k}.
\eeq
Using equation \eqref{def_Bk_Ek2} it is easy to see that
\beq\label{q_k_tau_formula}
q_k(\tau)=\sum\limits_{j=0}^k \tilde E_{2k-2j}\frac{1}{j!}\left(\frac{\i \tau}{\pi}\right)^j.
\eeq 
We introduce the sequence of functions $\{p_k(x)\}_{k\ge 0}$, defined by
\beq\label{def_pk_x}
p_k(x):=\int_0^1 e^{-xu} u^k \d u.
\eeq
For $k\ge 1$, $\tau>0$ and $z\in \r$ we define
\beq\label{def_big_Hn_z_tau}
H_k(z,\tau):=\frac{e^{\frac{\pi \i}{4}}}{\sqrt{\tau}} 
\sum\limits_{l=0}^{k-1} 
(-1)^le^{\frac{\pi \i}{\tau}(z+l+\frac{1}{2})^2} 
\left[1-\Phi\left(\sqrt{\tfrac{\pi}{\tau}}e^{\frac{\pi \i}{4}}(z+l+\tfrac{1}{2}) \right) \right],
\eeq
where $\Phi(z)$ is the error function
\beq\label{series_for_Phi}
\Phi(z):=\frac{2}{\sqrt{\pi}} \int_0^z e^{-x^2} \d x=\frac{2}{\sqrt{\pi}}
\sum\limits_{n\ge 0} \frac{(-1)^n}{n!} \frac{z^{2n+1}}{2n+1}, \;\;\; z\in \c. 
\eeq
Finally, for $z>0$ and $\tau \in \r$  we define
\beq\label{def_big_Jn_z_tau}
J(z,\tau):=\int_{0}^{\infty} 
e^{-2x z} f_{\tau}(x)  \d x,
\eeq
where $f_{\tau}(x)$ is given by \eqref{def_f_tau_z}.

\begin{proposition}\label{prop_compute_h_tau_z}\mbox{}
\begin{itemize}
\item[(i)] For $k\ge 1$, $|z|\le 1/2$ and $\tau \in (0,1)$ 
\beq\label{hnztau_computing}
h(z,\tau)=H_k(z,\tau)+H_k(-z,\tau)+\frac{(-1)^k}{\pi} 
\left(J(k+z,\tau)+J(k-z,\tau) \right).
\eeq
\item[(ii)]
For $\epsilon>0$ let us denote $K=K(\epsilon)=2+2 \lceil \ln(\tfrac{1}{\epsilon})\rceil$. Then
for any $\epsilon \in (0,\tfrac{1}{10})$, $|z| \le 1/2$ and $\tau \in (0,1)$  
\beq\label{formula_dzj_h_z_tau}
h(z,\tau)=H_K(z,\tau)+H_K(-z,\tau)+\frac{1}{\pi} 
\sum\limits_{0\le l \le K} q_l(\tau) \big (p_{2l}(2(K+z))+p_{2l}(2(K-z)) \big )+{\mathcal E},
\eeq
where $|{\mathcal E}|<\epsilon$.
\end{itemize}
\end{proposition}
\begin{proof}
Let us prove part (i). 
We denote $\theta:=e^{\frac{\pi \i}{4}}$, use the second integral representation in \eqref{def_hztau_2} and the identity
\beqq
\cosh(\pi \theta y)^{-1}=2\sum\limits_{l=0}^{k-1} (-1)^l e^{-(2l+1)\pi  \theta y}+(-1)^k e^{-2\pi k  \theta y}\cosh(\pi  \theta y)^{-1},
\eeqq
and obtain
\beq\label{prop_proof1}
h(z,\tau)=4\theta\sum\limits_{l=0}^{k-1} (-1)^l \int_0^{\infty} 
e^{-\pi \tau y^2-(2l+1)\pi \theta y} \cosh(2\pi z \theta y)\d y+
2\theta(-1)^k \int_0^{\infty} e^{-\pi \tau y^2 -2\pi k \theta y}
\frac{\cosh(2\pi z \theta  y)}{\cosh(\pi  \theta y)} \d y.
\eeq
The integrals in the sum in the right-hand side of 
\eqref{prop_proof1} can be evaluated explicitly by applying formula (3.322.2) in \cite{Jeffrey2007}; this gives us the terms 
$H_k(z,\tau)+H_k(-z,\tau)$. To deal with 
the remaining integral in the right-hand side of
\eqref{prop_proof1}, we rotate the contour of integration $\r^+ \mapsto e^{-\frac{\pi \i}{4}} \r^+$ and 
change the variable of integration $y=x/(\pi \theta)$. This gives us
\beqq
2\theta\int_0^{\infty} e^{-\pi \tau y^2 -2\pi k \theta y}
\frac{\cosh(2\pi z \theta  y)}{\cosh(\pi  \theta y)} \d y=\frac{1}{\pi} 
\left(J(k+z,\tau)+J(k-z,\tau) \right),
\eeqq
and ends the proof of the identity \eqref{hnztau_computing} and part (i) of the proposition.

Let us prove part (ii). We write 
\beq\label{J_K_two_integrals}
 J(K+z,\tau)=I_1+I_2=\int_{0}^{1} 
e^{-2x(K+z)} f_{\tau}(x)  \d x+
\int_{1}^{\infty} 
e^{-2x(K+z)} f_{\tau}(x)  \d x.
\eeq
Using the fact that $|z|\le 1/2$ and $|f_{\tau}(x)|<1$ for $x>0$, the second integral in the right-hand side of \eqref{J_K_two_integrals} can be  bounded from above as
\beq\label{estimate_second_integral}
\left \vert I_2 \right \vert
&<&\int_{1}^{\infty} 
e^{-x(2K-1)} \d x=\frac{e^{-2K+1}}{2K-1}< \frac{\epsilon}{4}.
\eeq

In order to deal with the first integral in the right-hand side of \eqref{J_K_two_integrals}, we expand $f_{\tau}(x)$ in Taylor series
in $x$ (which converges for $|x|<\pi/2$, see \eqref{def_f_tau_z}) and obtain
\beq\label{tail_series}
I_1= \sum\limits_{l\ge 0} q_l(\tau) 
\int_{0}^{1} e^{-2x(K+z)} x^{2l} \d x=
\sum\limits_{l\ge 0} q_l(\tau) p_{2l}(2(K+z)).
\eeq
In order to estimate the tail of the above series, first we will need to establish an upper bound for $q_k(\tau)$.
From the following identity for Euler numbers (see formulas 9.652.3 and 9.655.3 in \cite{Jeffrey2007})
\beqq
\tilde E_{2n}=\frac{E_{2n}}{(2n)!}=2(-1)^n \left(\frac{2}{\pi} \right)^{2n+1} \sum\limits_{k\ge 0} \frac{(-1)^k}{(2k+1)^{2n+1}}
\eeqq
we conclude that $|\tilde E_{2n}|<2\left(\frac{2}{\pi} \right)^{2n+1}$. Using formula \eqref{q_k_tau_formula} and the fact that $|\tau|<1$ 
we obtain
\beqq
|q_k(\tau)| \le \sum\limits_{j=0}^k |\tilde E_{2k-2j}|\frac{1}{j!}\left(\frac{|\tau|}{\pi}\right)^j
<\sum\limits_{j=0}^k 2\left(\frac{2}{\pi} \right)^{2k-2j+1} \frac{1}{j!}\left(\frac{1}{\pi}\right)^j<
2\left(\frac{2}{\pi} \right)^{2k+1} e^{\frac{\pi}{4}}< 3\left(\frac{2}{\pi} \right)^{2k}.
\eeqq
Using the above result we estimate the tail of the series in \eqref{tail_series} as follows
\beq\label{estimate_tail}
\left \vert \sum\limits_{l > K} q_l(\tau) p_{2l}(2(K+z)) \right \vert <
3 \sum\limits_{l > K} \left(\tfrac{2}{\pi} \right)^{2l}<\frac{\epsilon}{4}.
\eeq
Formulas \eqref{J_K_two_integrals}, \eqref{estimate_second_integral}, \eqref{tail_series} and \eqref{estimate_tail} 
show that for all $|z|\le 1/2$ and $\tau \in (0,1)$ 
\beqq
 J(K+z,\tau)=
 \sum\limits_{0 \le l \le K} q_l(\tau) p_{2l}(2(K+z))+ {\mathcal E_1},
\eeqq
where $|{\mathcal E}_1|<\frac{\epsilon}{2}$. The above statement combined with  \eqref{hnztau_computing} gives us the desired result \eqref{formula_dzj_h_z_tau}.
\end{proof}

\begin{proposition}\label{prop_Phi}
There exists an algorithm such that for any  $\epsilon \in (0,\tfrac{1}{10})$ and $x \in \r$ the value of 
$\Phi(e^{\frac{\pi i}{4}} x)$ can be evaluated to within $\pm \epsilon$ using 
$\le A_1 \ln\left(\tfrac{1}{\epsilon}\right)$ arithmetic operations on numbers of $\le A_2 \ln\left(\tfrac{1}{\epsilon}\right)$ bits.
The algorithm requires $\le A_3 \ln\left(\tfrac{1}{\epsilon}\right)$ bits of memory. 
\end{proposition}
\begin{proof}
Since $\Phi(e^{\frac{\pi i}{4}} x)$ is an odd function, we can restrict $x$ to be a positive number. 
When $x \in (0,1]$ we compute $\Phi(e^{\frac{\pi i}{4}} x)$ using Taylor series 
\eqref{series_for_Phi}. Since this series is converging exponentially fast, it can be truncated after $O(\ln(\tfrac{1}{\epsilon}))$ terms in 
order to achieve accuracy $\pm \epsilon$. 

Let us consider the case when $x \in (1,\infty)$. We will use the following result (see equations (3)-(6) in \cite{Hunter_Regan}):
For $z\in \c$ with $\re(z)>0$ 
\beq\label{hunter_regan_n1}
\Phi(z)=1-\frac{hze^{-z^2}}{\pi} \sum\limits_{k=-\infty}^{\infty} \frac{e^{-k^2 h^2}}{z^2+k^2h^2}+
R(z,h)+E(z,h),
\eeq
where $h>0$ and $h \re(z) \ne \pi$. Here
$R(z,h):=2(e^{\frac{2\pi z}{h}}-1)^{-1}$ if $\re(z)<\frac{\pi}{h}$ and $R(z,h):=0$ otherwise.
The error term $E(z,h)$ can be bounded from above as
\beq\label{hunter_regan_n1}
|E(z,h)|\le \frac{2 |ze^{-z^2}| e^{-\frac{\pi^2}{h^2}}}{\sqrt{\pi} (1-e^{-2\frac{\pi^2}{h^2}}) \big|(\re(z))^2-\frac{\pi^2}{h^2}\big|}.
\eeq
Let us define $\gamma:=\sqrt{\ln\left(\frac{1}{\epsilon}\right)}\;$, $z:=e^{\frac{\pi \i }{4}}x$ and 
\beq\label{def_h_x}
h:=
\begin{cases}
&\frac{\pi}{2\gamma}, \;\;\; \textnormal{if} \;\; x \le  2\gamma \; {\textnormal{or}} \; x \ge 4\gamma, \\
&\frac{\pi}{4\gamma},  \;\;\; \textnormal{if} \;\; x \in (2\gamma,4\gamma).
\end{cases}
\eeq 
This choice of $h$ implies $e^{-\frac{\pi^2}{h^2}} \le \epsilon^4$ and
\beqq
\frac{|z|}{\big|(\re(z))^2-\frac{\pi^2}{h^2}\big|}=\frac{x^{-1}}{\big |\frac{1}{2}-\frac{\pi^2}{x^2h^2}\big|}
<4,
\eeqq
therefore 
\beqq
|E(z,h)|<\frac{8 \epsilon^4}{(1-\epsilon^8)}< \frac{\epsilon}{2}.
\eeqq
Next, we define $N=\lceil 4 \gamma^2 \rceil$ and we estimate the tail of the series in \eqref{hunter_regan_n1} as 
\beqq
\left|h ze^{-z^2} \sum\limits_{n\ge N+1} \frac{e^{-k^2h^2}}{z^2+k^2h^2} \right |
< h \sum\limits_{n \ge N+1} khe^{-k^2h^2}<\int_{Nh}^{\infty} ue^{-u^2} \d u=\frac{1}{2} e^{-(Nh)^2}<\frac{\epsilon^4}{2}<\frac{\epsilon}{4}.
\eeqq
The above results show that for every $x>1$ we can choose $h$ according to \eqref{def_h_x} and obtain
\beq\label{Phi_final_formula}
\Phi\left(e^{\frac{\pi \i}{4}} x\right)=1-\frac{h e^{\frac{\pi \i}{4}} x e^{-\i x^2}}{\pi} \left(-\frac{\i}{x^2} + 2\sum\limits_{k=1}^{N} \frac{e^{-k^2 h^2}}{\i x^2+k^2h^2}\right)+{\mathcal E},
\eeq
where $|{\mathcal E}|<\epsilon$. Since the number of terms in the above sum is $N=\lceil 4 \ln\left(\frac{1}{\epsilon}\right) \rceil$, this ends the proof of
Proposition \eqref{prop_Phi} in the case $x \in (1,\infty)$. 
\end{proof}

\begin{proposition}\label{prop_compute_h_z_tau2}
There exists an algorithm such that for any  $\epsilon \in (0,\tfrac{1}{10})$, $|z|<10$ and $\tau \in (0,1)$ the value of  $h(z,\tau)$ can be evaluated to within $\pm \epsilon/\sqrt{\tau}$ using 
$\le A_4 \ln\left(\tfrac{1}{\epsilon} \right)^2$ arithmetic operations on numbers of $\le A_5 \ln\left(\tfrac{1}{\epsilon} \right)$ bits. The algorithm requires $\le A_6 \ln\left(\tfrac{1}{\epsilon}\right)^2$ bits of memory. 
\end{proposition}
\begin{proof}
All computations will be performed on numbers of $A_3 \ln\left(\tfrac{1}{\epsilon}\right)$ bits, where $A_3$ is the constant from Proposition \ref{prop_Phi}. We set $K=2+\lceil 2 \ln(\tfrac{1}{\epsilon})\rceil$.  
The first step is to pre-compute and store in the memory the values of $\tilde E_k$ for $0\le k \le 2K$. These numbers can be computed recursively (via formula 9.631 in \cite{Jeffrey2007}), this computation will require $O(K^2)$ arithmetic operations and $O(K^2)$ bits of memory. 
The second step is to use identity \eqref{hztau_period_1} and to normalize $z$ so that $|z|\le 1/2$ (note that 
this will require $O(|z|)$ arithmetic operations -- we will need this fact later). The thid step is to apply
formula \eqref{formula_dzj_h_z_tau}. According to \eqref{def_big_Hn_z_tau} and Proposition \ref{prop_Phi}, the computation of $H_K(\pm z, \tau)$
to the accuracy of $\pm \epsilon/\sqrt{\tau}$ can be achieved in $O(K^2)$ arithmetic operations using $O(K)$ bits of memory. 
We claim that the computation of the finite sum in \eqref{formula_dzj_h_z_tau} to the accuracy $\pm \epsilon$ can also be done in $O(K^2)$ arithmetic operations using $O(K)$ bits of memory (provided that we are using the pre-computed values of $\tilde E_k$). 
This follows from the fact that the coefficients $q_k(\tau)$ can be computed via \eqref{q_k_tau_formula} in $O(k)$ computations, 
while the values of $p_k(x)$ can be evaluated recursively via the identity
\beqq
p_k(x)=\frac{1}{x}\left(kp_{k-1}(x)-e^{-x}\right), \;\;\; k\ge 1, 
\eeqq
which follows easily from \eqref{def_pk_x} by integration by parts. Note that the above recurrence identity is numerically stable as long as $|x| \ge k$, which is true in formula \eqref{formula_dzj_h_z_tau}.
\end{proof}

\begin{proposition}\label{proposition_small_tau}
There exists an algorithm such that for any integer $n\ge 1$, $\epsilon \in (0,\tfrac{1}{10})$, $|z|\le 1/2$ and $|\tau| < n^{-4}$  the value of the function $F_n(z,\tau)$ can be evaluated  
to within $\pm \epsilon$ using 
$\le A_7 \ln\left(\tfrac{n}{\epsilon} \right)^3$ arithmetic operations on numbers of $\le A_8 \ln\left(\tfrac{n}{\epsilon} \right)$ bits.
The algorithm requires $\le A_9 \ln\left(\tfrac{n}{\epsilon}\right)^2$ bits of memory.  
\end{proposition}
\begin{proof}


The main idea behind this algorithm is to expand the exponential function in Taylor series, however the details of the implementation will
be different depending on whether $|z|>n^{-1}$ or $|z|<n^{-1}$. 
Let us consider the first case, when $|z|>n^{-1}$. We expand $\exp(2\pi \i \tau k^2)$ in Taylor series and obtain
\beqq
F_n(z,\tau)
 =\sum\limits_{l=0}^{\infty} \frac{(2\pi \i)^l}{l!} \left[
 \sum\limits_{k=0}^n  (\tau k^2)^l e^{2\pi \i z k} \right].
\eeqq
Since $|\tau| <n^{-4}$, the absolute value of the sum in the square brackets is bounded from above by $(n+1)$. Therefore, the external sum in $l$ is converging exponentially fast, and in order to achieve accuracy $\pm\epsilon$ we can truncate it after $L \le A_{10}\ln(\tfrac{n}{\epsilon})$ terms for some constant $A_{10}$. We assume that $L^3 < n$, otherwise we will compute $F_n(z,\tau)$ by direct summation. 
Our goal is to show that the sum in the square brackets can be evaluated with accuracy $\pm\epsilon$ in 
$O(L^2)$ operations on numbers of $O(L)$ bits using $O(L^2)$ bits of memory. 
The main idea is to rewrite this sum as follows
\beqq
\sum\limits_{k=0}^n  (\tau k^2)^l e^{2\pi \i z k}&=&
 \frac{\tau^l}{(2\pi \i)^{2l}} \times  \frac{\d^{2l}}{\d z^{2l}} \sum\limits_{k=0}^n  e^{2\pi \i z k}
=
 n\frac{(\tau n^4)^l}{(2\pi \i)^{2l}}  \times n^{-4l-1}
\frac{\d^{2l}}{\d z^{2l}} \left[ \frac{e^{2\pi \i z(n+1)}-1}{e^{2\pi \i z}-1}\right].
\eeqq
Using Leibniz rule we find that
\beq\label{Leibniz_1}
n^{-4l-1}
\frac{\d^{2l}}{\d z^{2l}} \left[ \frac{e^{2\pi \i z(n+1)}-1}{e^{2\pi \i z}-1}\right]
=\sum\limits_{j=0}^{2l} \binom{2l}{j} f_j(z) g_{2l-j}(z), 
\eeq
where we have defined 
\beqq
 f_j(z):=n^{-2j-1}\frac{\d^j}{\d z^j}\left[\frac{1}{e^{2\pi \i z}-1}\right],
\eeqq
and
\beqq
 g_j(z):=n^{-2j} \frac{\d^{j}}{\d z^{j}} \left[ e^{2\pi \i z(n+1)}-1 \right]=
 n^{-2j}\left((2\pi \i (n+1))^j e^{2\pi \i z(n+1)}-{\bf 1}_{\{j=0\}}\right).
\eeqq
While the computation of $g_j(z)$ is straighforward, the computation of $f_j(z)$ requires more work. First, we check by induction that 
\beq\label{compute_f_j}
f_j(z)=\sum\limits_{k=1}^{j+1} a_{j,k} (n(\exp(2\pi \i z)-1))^{-k}, 
\eeq
where $a_{0,1}=1$ and the remaining coefficients $a_{j,k}$ can be computed by the recursion 
\beq\label{recursion_a_k_i}
a_{j+1,k}=-\frac{2\pi \i}{n} \left((k-1)a_{j,k-1}+ \frac{k}{n} a_{j,k} {\bf 1}_{\{k\le j+1\}}\right), \;\;\; j\ge 1, \; 1\le k \le j+2. 
\eeq
From \eqref{recursion_a_k_i} one can see by induction that $|a_{j,k}|<(4\pi j/n)^j<1$ (recall that $j<2L<2 n^{\frac{1}{3}}$). 
Note that the condition $n^{-1}<|z|\le 1/2$ implies 
\beq\label{bound_sin_cos}
n|\exp(2\pi \i z)-1|>\max(n|\sin(2\pi z)|,n|\cos(2\pi z)-1|) \ge 4,
\eeq
since $|\sin(2\pi z)| \ge 4|z|$ if $|z| \le 1/4$ and $|\cos(2\pi z)-1|>1$ if $1/4<|z|<1/2$. Given \eqref{bound_sin_cos}
and the above upper bound on the coefficients $a_{j,k}$, it is clear that formula \eqref{compute_f_j} provides a numerically stable way of computing $f_j(z)$ using numbers of 
$O(L)$ bits. The memory requirement is $O(L^2)$ bits, since in order to compute the coefficients $a_{j+1,k}$, $1\le k \le j+2$ via \eqref{recursion_a_k_i} we need to store at most $2L+1$ numbers $a_{j,k}$, $1\le k \le j+1$. 

When $|z|<n^{-1}$ the lower bound \eqref{bound_sin_cos} is no longer valid, and we have to proceed by a different route. In this case
we expand the exponential function $\exp(2\pi \i  (zk+\tau k^2))$ in Taylor series and obtain
\beq\label{F_n_small_z_small_tau}
F_n(z)=\sum\limits_{l\ge 0} \frac{(2\pi \i)^l}{l!} \left[\sum\limits_{k=0}^n (zk+\tau k^2)^l \right].
\eeq
Since $|z|<n^{-1}$ and $|\tau| \le n^{-4}$, the sum in the square brackets is bounded from above by $(n+1)2^l$. Therefore, the sum in $l$ is converging exponentially fast, and in order to achieve accuracy $\pm\epsilon$ we can truncate it after $L=O\left(\ln\left(\tfrac{n}{\epsilon}\right)\right)$ terms. 
The sum in the square brackets in \eqref{F_n_small_z_small_tau} can be computed as follows
\beqq
\sum\limits_{k=0}^n (zk+\tau k^2)^l=\sum\limits_{j=0}^l \binom{l}{j} (nz)^{l-j} (n^2\tau)^{j} S_{l+j}(n),
\eeqq
where we have defined $S_{j}(n):=n^{-j}\sum_{k=0}^n k^j$. Formula (9.623.1) in \cite{Jeffrey2007} gives us
\beqq
S_j(n)=\frac{n}{j+1} \sum\limits_{i=0}^j (-1)^i \binom{j+1}{i} B_i n^{-i}, \;\;\; j\ge 1, \; n\ge 1,
\eeqq
where $B_i$ are the Bernoulli numbers $\{1,-1/2,1/6,\dots\}$. We will leave it to the reader to verify that the above 
three formulas provide the required algorithm for computing $F_n(z,\tau)$ to within $\pm \epsilon$ in $O(L^3)$ arithmetic operations on numbers of $O(L)$ bits (one should precompute and store $2L$ values of $B_i/i!$, $0\le i \le 2L$, which can be done in $O(L^2)$ arithmetic operations using $O(L^2)$ bits of memory). 
\end{proof}

\vspace{0.1cm}
\noindent
{\bf Proof of Theorem \ref{thm_main2}:}
 We are given $z \in (0,1)$, $\tau \in (0,1)$, a positive integer $n$ and a small
positive number $\epsilon$.  We will describe the algorithm for computing the value of $F=F_n(z,\tau)$.  In order to start the algorithm,
 we use identities \eqref{identities_step_1} and normalize $z$ and $\tau$ so that $|z| \le 1/2$ and $|\tau| \le 1/4$; we define $z_1$ and $\tau_1$ to be
 equal to these normalized values of $z$ and $\tau$. The algorithm is based on a recursion, and $j$ will be the counter which keeps track of
 the steps of the recursion. We initialize $j=1$, $n_1=n$, $\alpha_1=1$ and $\beta_1=0$. 
 All computations are performed on numbers of $\lceil \max(5 A_5,3 A_8)\ln\left(\tfrac{n}{\epsilon}\right) \rceil$ bits, where $A_5$ and $A_8$ are the constants appearing in Propositions \ref{prop_compute_h_z_tau2} and \ref{proposition_small_tau}.  
 
\vspace{0.1cm}
\noindent
{\bf The algorithm:} 
\begin{itemize}
 \item[(i)]  If $n_j \le \ln(n)^3$ then we compute $f=F_{n_j}(z_j,\tau_j)$ by direct summation. Terminate the algorithm and return $F=\alpha_j f + \beta_j$. 
 \item[(ii)] If $|\tau_j| < n_j^{-4}$ then we compute $f=F_{n_j}(z_j,\tau_j)$ to the accuracy of $\pm \epsilon/n^3$ using the algorithm from Proposition \ref{proposition_small_tau}. Terminate the algorithm and return $F=\alpha_j f + \beta_j$.
\item[(iii)] If $|\tau_j| \ge n_j^{-4}$ we set $n_{j+1}=\lfloor 2 n_j |\tau_j| \rfloor$, 
$\tilde z=\frac{z_j}{2|\tau_j|}$ and $\tilde\tau=-\frac{1}{4\tau_j}$.
Set $z_{j+1}$ and $\tau_{j+1}$ to be the normalized values of $\tilde z$ and $\tilde \tau$ (use identities \eqref{identities_step_1}).
If $\tau_j>0$, then 
\beq\label{formula_alpha_beta1}
\alpha_{j+1}=\frac{\alpha_j}{\sqrt{2\tau_j}} \exp\left(\frac{\pi \i}{4}-\frac{\pi \i z_j^2}{\tau_j}\right), \;\;\;
\beta_{j+1}=\beta_j+\alpha_jR_{n_{j+1},n_j}(z_j,\tau_j),
\eeq
while if $\tau_j<0$ then 
\beq\label{formula_alpha_beta2}
\alpha_{j+1}=\frac{\alpha_j}{\sqrt{2|\tau_j|}} \exp\left(-\frac{\pi \i}{4}-\frac{\pi \i z_j^2}{\tau_j}\right), \;\;\;
\beta_{j+1}=\beta_j+\alpha_j{\overline{R_{n_{j+1},n_j}(-z_j,|\tau_j|)}},
\eeq
where $R_{m,n}(z,\tau)$ is given by by \eqref{def_Rmnztau}, and the values of the Mordell integral $h(\cdot,\cdot)$ appearing 
in \eqref{def_Rmnztau} are computed using the algorithm from Proposition \ref{prop_compute_h_z_tau2} to the accuracy $\pm \epsilon/n^3$. 
\item[(iv)] Increase the counter $j\mapsto j+1$ and proceed to step (i). 
\end{itemize}
Assume that this algorithm stops after $J$ iterations. The fact that this algorithm returns the correct value of $F_n(z,\tau)$ can be verified by induction on $J$ using identity \eqref{main_formula}. 
Let us investigate the number of arithmetic operations required by this algorithm. At each iteration of the algorithm, provided that it does not terminate in steps (i) or (ii), 
 we have the new value $n_{j+1}$ which satisfies $n_{j+1}=\lfloor 2 n_j |\tau_j| \rfloor \le n_j/2$ 
 (recall that $|\tau_j| \le 1/4$). This shows that the algorithm will either terminate in step (i) after at most $\log_2(n)$ iterations, or it 
 will terminate in step (ii) before that. 
Let us denote $L=\ln\left(\tfrac{n}{\epsilon}\right)$. Step (ii) (resp. step (iii)) requires $O(L^3)$ (resp. $O(L^2)$) 
arithmetic operations and both of these steps require $O(L^2)$ bits of memory (see Propositions \ref{prop_compute_h_z_tau2}
and  \ref{proposition_small_tau}). Since step (ii) will be executed at most once, and step (iii) at most $\log_2(n)$ times, it is clear that 
the algorithm requires $O(L^3)$ arithmetic operations and $O(L^2)$ bits of memory. 

Finally, let us consider the accuracy of this algorithm. All numbers appearing in the algorithm are evaluated to the accuracy $\pm \epsilon/n^3$. 
Since the algorithm did not terminate at the iteration $J-1$, we have 
$n_{J-1}>1$ and $|\tau_{J-1}| \ge n_{J-1}^{-4}>n^{-4}$. Recall that for all $j$ we have $n_{j+1} \le 2 n_j |\tau_j|$, therefore
\beqq
1<n_{J-1} \le 2^{J-2}n\prod\limits_{i=1}^{J-2} |\tau_i|,
\eeqq
and applying formulas \eqref{formula_alpha_beta1} and \eqref{formula_alpha_beta2} we obtain
\beqq
|\alpha_{J-1}|=\left[2^{J-2}\prod\limits_{i=1}^{J-2} |\tau_i| \right]^{-\frac{1}{2}} < \sqrt{n}.
\eeqq 
Assuming that the algorithm terminates in step (ii), then the final accuracy is at least $(\pm \epsilon/n^3)\times \sqrt{n} \times \log_2(n)$, which is smaller than the required accuracy $\pm \epsilon$. On the other hand, if the algorithm terminates in step (i), 
then $|\alpha_J|=|\alpha_{J-1}|/\sqrt{2|\tau_{J-1}|}< n^{\frac{5}{2}}$, and the final accuracy is $(\pm \epsilon/n^3)\times n^{\frac{5}{2}} \times \log_2(n)$, which is still smaller than $\pm \epsilon$. 
\qed

\label{page11}

\vspace{0.25cm}
\noindent
{\bf Remark 1:} One may ask the following natural question: is the choice $n_{j+1}=\lfloor 2 n_j |\tau_j| \rfloor$ in the above algorithm  optimal? In other words, given that the identity \eqref{main_formula} is true for all positive $m$ and $n$, why can not we choose $m=n_{j+1} \ll 2 n_j |\tau_j|$, thus reducing the number of terms in the new sum $F_{m}(\cdot,\cdot)$? The rationale for this choice is that
the computation $R_{m,n_j}(\cdot,\cdot)$ in formula \eqref{def_Rmnztau} requires the evaluation of the Mordell integral  
$$h(z_j+(2n_j+1)|\tau_j|-m-\tfrac{1}{2},-2\tau_j),$$ 
which involves more than $|2n_j|\tau_j| -m|$ arithmetic operations (see step 2 in the proof of Proposition \ref{prop_compute_h_z_tau2}). Therefore, while the choice of $m=n_{j+1} \ll 2n_j |\tau_j|$ will decrease the number of terms (and the computation time) of $F_{m}(\cdot,\cdot)$, 
any gain will be canceled by the corresponding increase in the computation time of $R_{m,n_j}(\cdot,\cdot)$.

\section{Practical implementation and extensions of the algorithm}\label{section_numerics}

As is often the case, the algorithm which can be analyzed analytically and which allows for rigorous error bounds is not necessarily the most efficient algorithm from the practical point of view. While it is certainly possible to perform practical computations 
of $F_n(z,\tau)$ using the algorithm described in the proof 
of Theorem \ref{thm_main2}, in this section we will explain how one could produce a  more efficient algorithm with a certain amount of pre-computation and a few numerical experiments. This practical implementation is suitable in the case when
we need to compute $F_n(z,\tau)$ to a fixed accuracy $\pm \epsilon$ for many different values of $z$, $\tau$ and $n \le N_1 $ (for 
some fixed value of $N_1$). 

As we saw in the proof of Theorem \ref{thm_main2}, the main computational effort is spent in evaluating $F_n(z,\tau)$ for very small values of $\tau$ (when $|\tau|<n^{-4}$) and in computing the Mordell integral $h(z,\tau)$. We do not see a way of making the former of these tasks much faster, however the latter task can certainly be done much more efficiently. Part (i) of Proposition \ref{prop_compute_h_tau_z} shows that in order to compute $h(z,\tau)$ we need to be able to evaluate the error function $\Phi(e^{\frac{\pi \i }{4}} x)$ for $x\in \r$ and to compute $J(z,\tau)$ defined by \eqref{def_big_Jn_z_tau}. Let us first discuss the computation of the error function. 
Our approach to computing  $\Phi(e^{\frac{\pi \i}{4}}x)$ is to divide the interval $(0,\infty)$ into a number of 
sub-intervals $0<x_1<x_2<\dots<x_m=x^*<\infty$, and use the Chebyshev approximation on each sub-interval. 
On the infinite interval $(x^*,\infty)$ we define the function $f(\sigma)$ via 
$\Phi(e^{\frac{\pi \i}{4}}x)=1+e^{-\i xz^2} x^{-1} f(\sigma)$, where
 $\sigma:=(x^*/x)^2$ and we approximate $f(\sigma)$ by the first few terms of the Chebyshev series. 
It is known (see \cite{Nemeth}) that the coefficients of the corresponding Chebyshev series decay as $\exp(-2\sqrt{n x^*}) O(n^{-\frac{1}{2}})$, therefore, by a proper choice of $x^*$ we can be sure that we need only a few terms of the Chebyshev series to obtain the required accuracy.
Once we have chosen $x^*$, we divide $(0,x^*)$ into $m$ sub-intervals of equal length, and on each of them we compute the first few terms of the Chebyshev series. Note that on each finite interval $(x_i,x_{i+1})$, $1\le i <m$, the Chebyshev series approximating $\Phi(e^{\frac{\pi \i}{4}}x)$ must converge exponentially fast since $\Phi(z)$ is an entire functions. By choosing $m$ large enough we can make sure that the number of significant terms in each Chebyshev series is small. For example, in our implementation of this algorithm we chose $x^*=5$ and $m=5$, and on each subinterval we approximated $\Phi(e^{\frac{\pi \i}{4}}x)$ by the first thirty terms of Chebyshev series. This approximation had absolute error $\le 10^{-30}$ over all real values of $x$.

\label{J_discussion}
The second important problem is how to evaluate $J(k+z,\tau)$ efficiently. 
Our approach is based on the following formula 
\beq\label{J_practical_computing}
J(k+z,\tau)=\frac{1}{2k}\int_{0}^{\infty} 
e^{-y} g_{z,\tau}\left(\tfrac{y}{k}\right) \d y, \;\;\;
{\textnormal{ where }} \;\;\;
 g_{z,\tau}\left(y\right):=\frac{e^{\frac{\i \tau y^2}{4\pi}-zy}}{\cosh(\frac{y}{2})},
\eeq
which follows from \eqref{def_big_Jn_z_tau} by changing the variable of integration $x=y/(2k)$.
For $k\ge 1$,  $|z|\le 1/2$ and $\tau\in (0,1)$ the function  $y \in (0,\infty) \mapsto g_{z,\tau}\left(\tfrac{y}{k}\right)$ is bounded, and as $k \to +\infty$ it converges to $g_{z,\tau}(0)=1$. Therefore, when $k$ is reasonably large, the integral in \eqref{J_practical_computing} can be computed to a very high-precision using the Gauss-Laguerre quadrature with the weight function $e^{-y}$ and $M$ nodes. 
The problem is to decide what ``reasonably large" means, and here one should do a number of numerical experiments to find the optimal values of $k$ and $M$.  If we take $k$ to be a large number, then the function $g_{z,\tau}\left(\tfrac{y}{k}\right)$ is very close to $1$, 
and $M$ -- the number of nodes in Gauss-Laguerre quadrature -- can be taken quite small in order to achieve the required accuracy. 
Of course, the disdvantage of choosing $k$ to be large 
is that it would require many evaluations of the error function in \eqref{def_big_Hn_z_tau}, and it would increase the run-time of the algorithm. On the other hand, if we take $k$ to be a small integer, then the function  $g_{z,\tau}\left(\tfrac{y}{k}\right)$ oscillates and $M$ has to be very large in order to provide the required accuracy, which would again increase the run-time of the algorithm. Therefore, $k$ 
and $M$ have to be chosen so that the computation time of $H_k(z,\tau)$ it approximately equal to the computation time of $J(k+z,\tau)$.

In our examples we took $k=5$ in formulas  \eqref{hnztau_computing}, \eqref{def_big_Hn_z_tau} and \eqref{def_big_Jn_z_tau}
 and we have used the Gauss-Laguerre quadrature with $M=124$ nodes (truncated to the smallest
 40 nodes, see \cite{Mastroianni}) to evaluate the integral in 
 \eqref{def_big_Jn_z_tau}. In order to verify the accuracy, we have computed $h(z,\tau)$ on a very fine regular grid of points in the rectangle $|z|\le 1/2$ and $\tau \in (0,1/2)$ using the above algorithm and we have checked that the relative error is always less than $10^{-29}$
 (when compared  with the algorithm described in Proposition \ref{prop_compute_h_z_tau2}). While the algorithm based on Gauss-Laguerre quadrature is very efficient,  we were not able to provide rigorous error estimates. It is known that the error of the $M$-point Gauss-Laguerre quadrature can be bounded by a multiple of 
\beqq
\eta_{2M}:=\sup\left\{ \; \left|\frac{\partial^{2M}}{\partial y^{2M}} g_{z,\tau}\left(y\right)\right| \; : \; y\ge 0\right\},
\eeqq   
(see Theorem 3 in \cite{Stroud_Chen}),  but we were not able to obtain good upper bounds for this quantity.

\begin{figure}[t]
\centering
\subfloat[][]{\label{fig2_simple}\includegraphics[height =5cm]{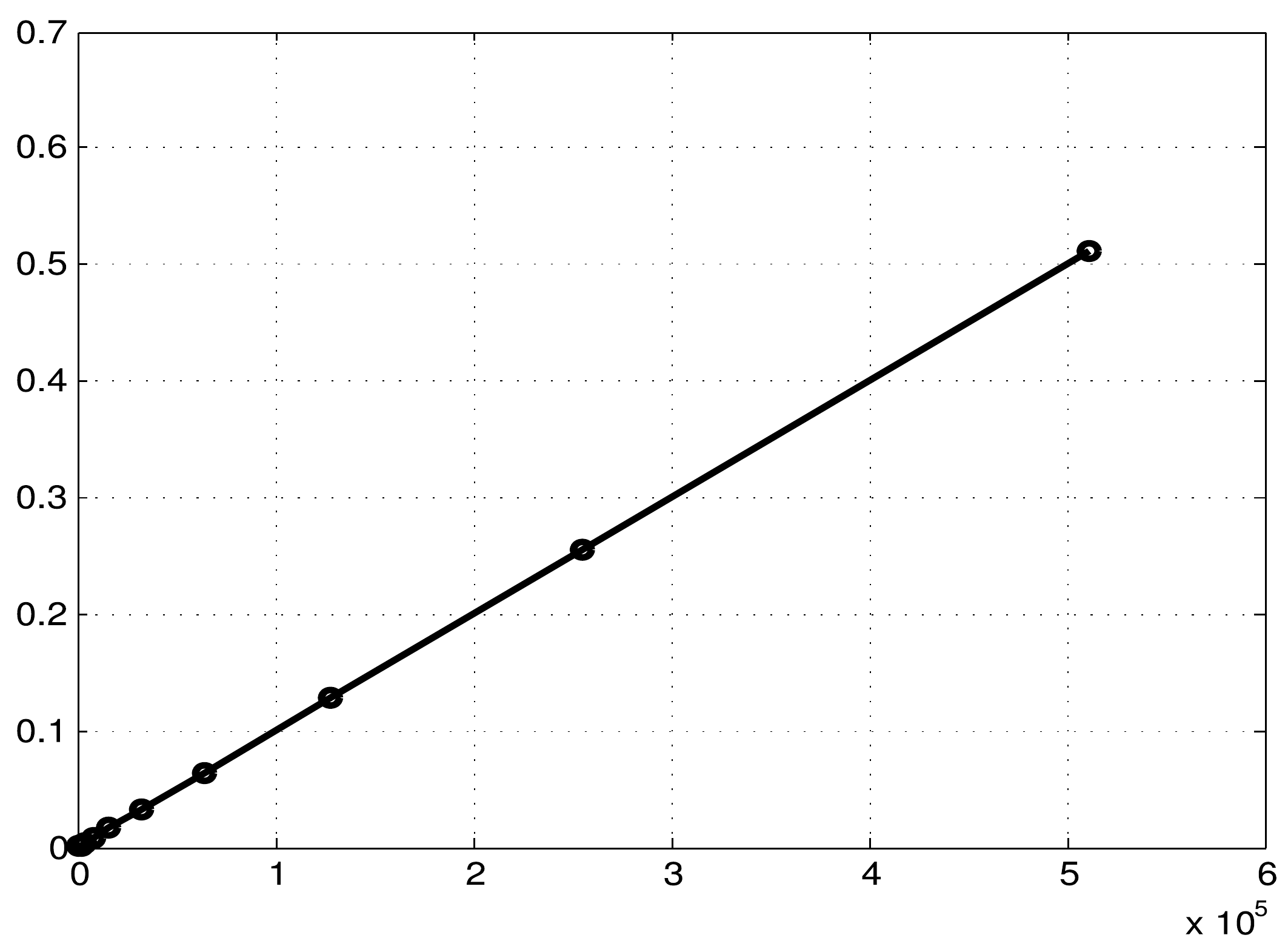}} \qquad 
\subfloat[][]{\label{fig2_fast}\includegraphics[height =5.1cm]{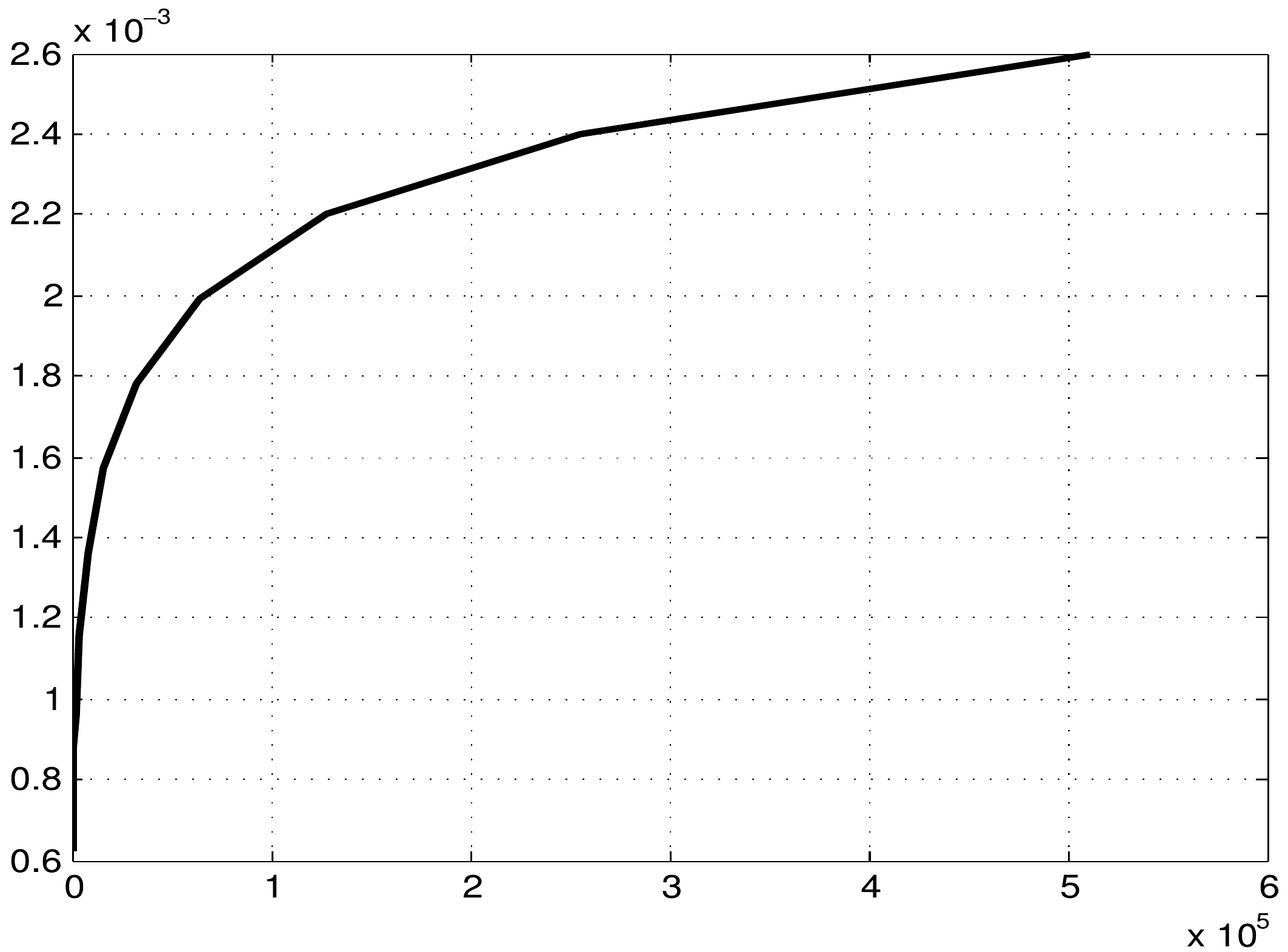}} \\
\subfloat[][]{\label{fig2_simple_vs_fast}\includegraphics[height =5cm]{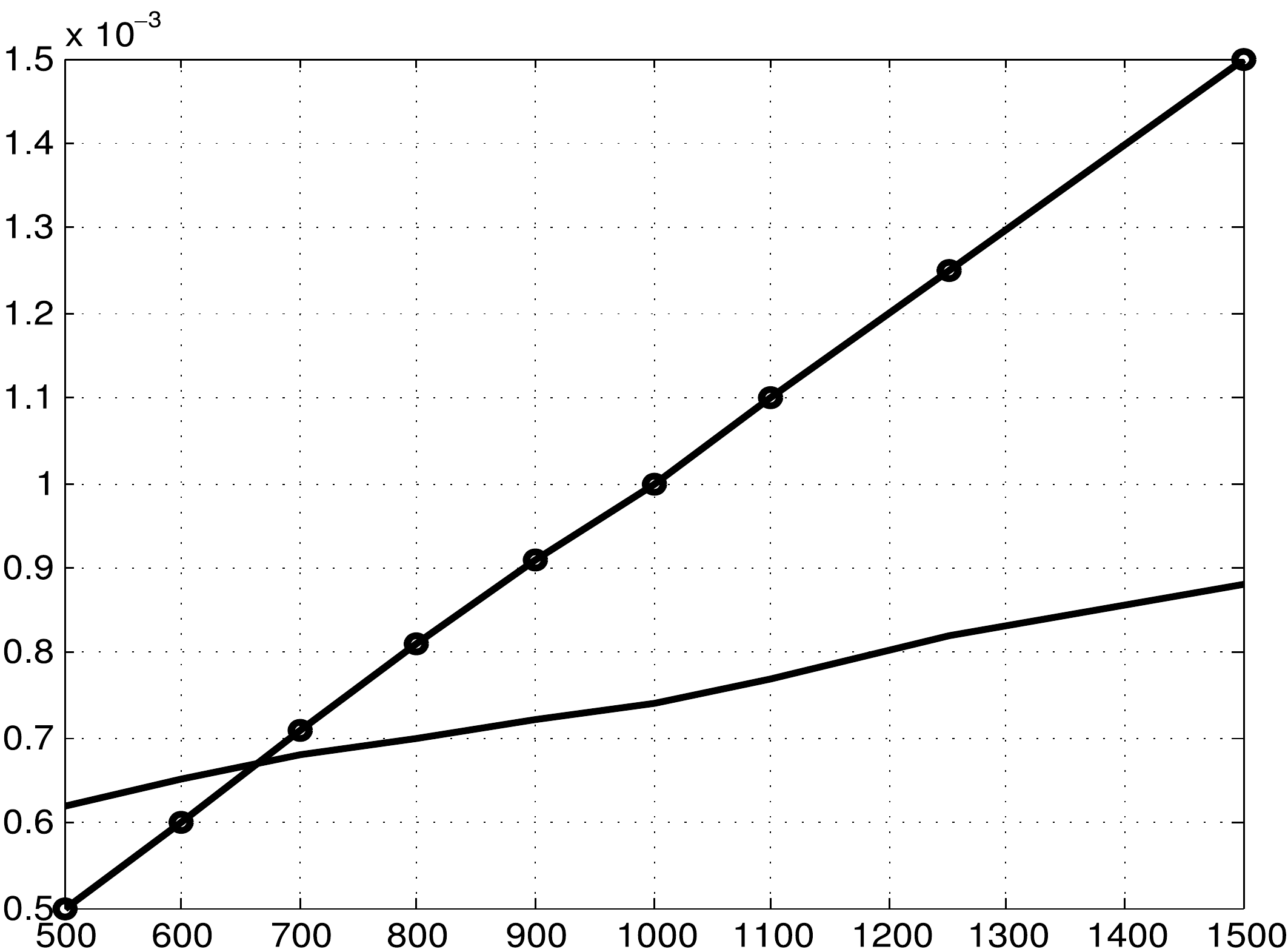}}  \qquad
\subfloat[][]{\label{fig2_fast_log_scale}\includegraphics[height =5.02cm]{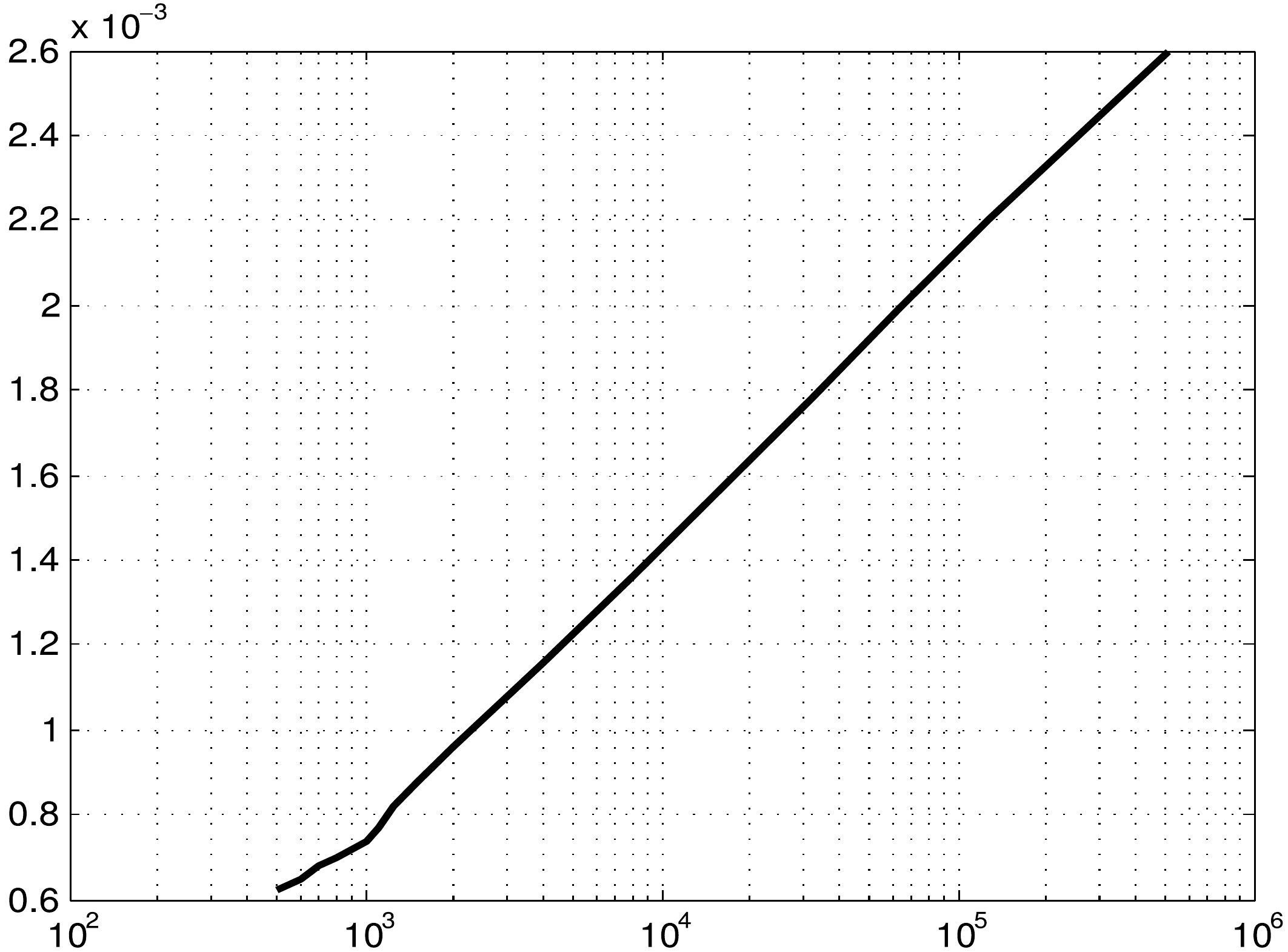}} 
\caption{Average run time for computing a single value of $F_n(z,\tau)$ ($x$-axis represents $n$, $y$-axis represents time in seconds).
The curve marked with circles corresponds to direct summation algorithm based on \eqref{def_F}, the curve without markers corresponds to Hiary's algorithm.  }
\label{fig2}
\end{figure}

The results of our numerical experiments are presented in Figure \ref{fig2}. The algorithm 
 was implemented in Fortran using the quadruple precision. Due to the fact that the number of iterations of the algorithm
 varies for different values of $\tau$,  we have tested the algorithm for 1000 random pairs 
 $(z,\tau)$, sampled uniformly from the domain $0<\tau<1/4$ and $-1/2<z<1/2$, and 
 the results in Figure \ref{fig2} represent the average run-time for a single evaluation
 of $F_n(z,\tau)$. Since we are working in fixed precision (which does not depend on $n$, as in the algorithm in the proof of 
 Theorem \ref{thm_main2}), these computations involve an unavoidable loss of precision, which becomes more pronounced as $n$ increases.
 See the paragraph preceding Remark 1 on page \pageref{page11}, explaining the reason for this loss of precision. 
 Our results indicate that the difference between the values of $F_n(z,\tau)$ computed 
 via direct summation \eqref{def_F} and our version of Hiary's algorithm is typically of the order of $10^{-28}$ for $n=10^3$ and around
$10^{-25}$ for $n=10^5$. This loss of precision is still acceptable for practical purposes. 
Most importantly, the figures \ref{fig2_fast} and \ref{fig2_fast_log_scale} confirm that the run-time 
of the algorithm increases essentially logarithmically with $n$, 
and that the simplified version of Hiary's algorithm is much faster than the 
summation of $n$ terms in \eqref{def_F} even for moderately large $n$.

Finally, we will discuss a related problem of evaluating the finite sum of the form
\beq\label{def_F_n_j}
F_{n,j}(z,\tau):=n^{-j} \sum\limits_{k=1}^n k^j e^{2\pi \i (zk+\tau k^2)}=
(2\pi \i n)^{-j} \frac{\partial^j}{\partial z^j} F_n(z,\tau). 
\eeq
Hiary's result  \cite[Theorem 1.1]{Hiary_theta} states that $F_{n,j}(z,\tau)$ can be computed to the required accuracy $\pm \epsilon$ in poly-log time in $n/\epsilon$ (though the precise statement is slightly more complicated, as the implied constants also depend on $j$). 
Our results can also be adapted to give a  simpler version of Hiary's poly-log time algorithm for computing $F_{n,j}(z,\tau)$. Below we will sketch the main ideas of the practical implementation of such an algorithm.  

We will follow the same path as in the algorithm described in the proof of Theorem \ref{thm_main2}. Our goal is to compute the values of 
$F_{n,j}(z,\tau)$ for $0\le j \le J$. At each step of the recursion, we normalize the values of $z$ and $\tau$ so that $|z|\le 1/2$ and $|\tau|\le 1/4$. If $|\tau|<n^{-4}$, then we compute $F_{n,j}(z,\tau)$ by expanding the exponential function in \eqref{def_F_n_j} in Taylor series and applying similar ideas as in the proof of Proposition \ref{proposition_small_tau}. 
If $|\tau|\ge n^{-4}$, then we use the following identities 
\beq\label{main_formula_derivatives}
\frac{\partial^j}{\partial z^j} F_n(z,\tau)=\frac{e^{\frac{\pi \i }{4}}}{\sqrt{2\tau}}
\sum\limits_{k=0}^j \binom{j}{k}
\left[
\frac{\partial^{j-k}}{\partial z^{j-k}} e^{-\frac{\pi \i z^2}{2\tau}} \right]
\times  
\left[\frac{\partial^k}{\partial z^k}F_m\left(\tfrac{z}{2\tau},-\tfrac{1}{4\tau}\right)\right]+
\frac{\partial^j}{\partial z^j}R_{m,n}(z,\tau), \;\;\; 0\le j \le J,
\eeq
which follow from \eqref{main_formula} by applying Leibniz rule. These identities reduce the computation of $F_{n,j}(z,\tau)$ to the computation of $F_{m,j}(\cdot, \cdot)$ with $m=\lfloor 2n\tau \rfloor<n/2$ and $1\le j \le J$, and complete the main step of the recursion.

Applying identities \eqref{main_formula_derivatives} in practice will involve computing $\frac{\partial^j}{\partial z^j}h(z,\tau)$,
which is equivalent to evaluating   $G_1=\frac{\partial^j}{\partial z^j}H_k(z,\tau)$ and 
$G_2=\frac{\partial^j}{\partial z^j}J(k+z,\tau)$ (see formulas \eqref{def_big_Hn_z_tau},  \eqref{def_big_Jn_z_tau} and \eqref{hnztau_computing}).
 Computing $G_1$ does not pose a serious problem, as applying Leibniz rule to \eqref{def_big_Hn_z_tau} would give us an explicit 
expression for  $G_1$, and  since $\Phi'(z)=(2/\sqrt{\pi})\exp(-z^2)$ it is easy 
to see that this explicit expression would involve only elementary functions and the error function $\Phi(z)$. 
At the same time, when $z$ is large it will be more efficient to compute  
$\frac{\partial^j}{\partial z^j} e^{z^2} (\Phi(z)-1)$ directly by taking derivatives of the right-hand side of formula \eqref{Phi_final_formula} (or by using the asymptotic expansion for this function, see formula 8.254 in \cite{Jeffrey2007}). 
The value of $G_2=\frac{\partial^j}{\partial z^j}J(k+z,\tau)$ can be computed using the generalized Gauss-Laguerre quadrature. 
Indeed, from \eqref{def_big_Jn_z_tau} we find that
\beq\label{J_derivatives}
\frac{\partial^j}{\partial z^j}J(k+z,\tau)= \frac{(-1)^j}{2} k^{-j-1} 
\int_{0}^{\infty} 
e^{-y} y^j g_{z,\tau}\left(\tfrac{y}{k}\right) \d y,
\eeq
where $g_{z,\tau}(y)$ is defined in \eqref{J_practical_computing}. 
Using the same strategy as discussed on page \pageref{J_discussion} (following equation \eqref{J_practical_computing})  the integral in the right-hand side of \eqref{J_derivatives} can be evaluated using the generalized Gauss-Laguerre quadrature with the weight function $x^j e^{-x}$ and $M$ nodes. By experimenting with different values of $k$ and $M$ (and choosing the optimal ones) we can 
compute $\frac{\partial^j}{\partial z^j}h(z,\tau)$ very efficiently with the required accuracy.



\end{document}